\documentclass[12pt,oneside,english,shortlabels]{amsart}

\usepackage{babel}
\usepackage{amstext}
\usepackage{amssymb}

\usepackage{amsmath} 
\usepackage{latexsym}
\usepackage{graphicx} 

\usepackage{mathptmx}

\usepackage{graphicx} 

\makeatletter

\usepackage{amsthm}
\usepackage{enumitem}		

\usepackage{mathrsfs} 
\providecommand{\noopsort}[1]{}

\numberwithin{equation}{subsection}

\theoremstyle{definition} 

 \newtheorem{definition}{Definition}[section]
 \newtheorem{remark}[definition]{Remark}

\newtheorem*{notation}{Notations}

\theoremstyle{plain}

 \newtheorem{theorem}[definition]{Theorem}
 
 \newtheorem{lemma}[definition]{Lemma}

\makeatletter
\newcommand*{\house}[1]{
  \mathord{
    \mathpalette\@house{#1}
  }
}
\newcommand*{\@house}[2]{
  \dimen@=\fontdimen8 %
      \ifx#1\scriptscriptstyle\scriptscriptfont
      \else\ifx#1\scriptstyle\scriptfont
      \else\textfont\fi\fi
      3 %
  \sbox0{%
    $#1%
      \vrule width\dimen@\relax
      \overline{%
        \kern2\dimen@
        \begingroup 
          #2%
        \endgroup
        \kern2\dimen@
      }
      \vrule width\dimen@\relax
      \mathsurround=1.5\dimen@ 
    $
  }
  \ht0=\dimexpr\ht0-\dimen@\relax
  \dp0=\dimexpr\dp0+2\dimen@\relax
  \vbox{
    \kern\dimen@ 
    \copy0 
  }
}

\def\11{{\mathbf 1}}

\theoremstyle{remark}
\newtheorem{exampl}[subsubsection]{Example}

\def\bee{\begin{exampl}}
\def\eee{\end{exampl}}
\def\bn{\begin{notation}}
\def\en{\end{notation}}
\def\br{\begin{remark}}
\def\er{\end{remark}}
\def\bp{\begin{prop}}
\def\ep{\end{prop}}
\def\bpr{\begin{proof}}
\def\epr{\end{proof}}
\def\bt{\begin{thm}}
\def\et{\end{thm}}
\def\be{\begin{equation}}
\def\ee{\end{equation}}
\def\bl{\begin{lem}}
\def\el{\end{lem}}
\def\bc{\begin{cor}}
\def\ec{\end{cor}}
\def\bd{\begin{defn}}
\def\ed{\end{defn}}

\numberwithin{equation}{subsection}

\author{Jaroslav Han$\check{{\rm C}}$l$^{\dag}$}
\thanks{}
\address{$\mbox{}^{\dag}$
Department of Mathematics,  
Faculty of Sciences, University of Ostrava, 
30. dubna  22, 701  
03 Ostrava 1,  Czech Republic}
\email{hancl@osc.cz}
\urladdr{}

\author{Radhakrishnan Nair$^{\ddag}$}
\thanks{}
\address{$^{\ddag}$
Mathematical Sciences, The University of Liverpool, Peach Street, 1 Peach Street, Liverpool, L69 7ZL, U.K.}
\email{nair@liv.ac.uk}
\urladdr{}

\author{Jean-Louis Verger-Gaugry$\mbox{}^{\Diamond}$ }
\thanks{}
\address{$\mbox{}^{\Diamond}$ 
LAMA, CNRS UMR 5127,
Univ. Grenoble Alpes, Univ. Savoie Mont~Blanc,
\newline
F - \!73000 Chamb\'ery, France}
\email{Jean-Louis.Verger-Gaugry@univ-smb.fr}
\urladdr{}

\title[On Polynomials in Primes, Ergodic Averages and Monothetic Groups]
{On Polynomials in Primes, Ergodic Averages and Monothetic Groups}

\begin{document}

\title{On Polynomials in Primes, Ergodic Averages and Monothetic Groups}

\maketitle

\begin{abstract}
Let $G$ denote a compact monothetic group, 
and  let 
$$\rho (x) = \alpha_k x^k + \ldots  + 
\alpha_1 x + \alpha_0,$$
where $\alpha_0, \ldots , \alpha_k$ are 
elements of  $G$ one of which is a generator of 
$G$.   Let  $(p_n)_{n\geq 1}$ 
denote the sequence 
of rational prime numbers.  
Suppose $f \in L^{p}(G)$ for $p> 1$.  
It is known that if
$$A_{N}f(x) := {1 \over N} \sum _{n=1}^{N}
f(x + \rho (p_n)) \qquad 
(N=1,2, \ldots ),$$  
then the limit 
$\lim _{n\to \infty} A_Nf(x)$ exists for almost 
all $x$ with 
respect Haar measure.  
We show that if $G$ is connected then the limit 
is $\int_{G} f d\lambda$.  
In the case where $G$ is 
the $a$-adic integers, which is a totally 
disconnected group, the limit is described in 
terms of Fourier multipliers which are 
generalizations of Gauss sums.
\end{abstract}

\vspace{0.7cm}

\vspace{0.5cm}


\tableofcontents  

\newpage
\section{Introduction}
\label{S1}

Topological groups, having a dense cyclic subgroup have been well studied.  For the class of locally compact groups, following D. Van Dantzig, who was the first to study them 
\cite{vandantzig}, these groups are called monothetic.  Such groups are fully 
classified in \cite{hewittross}.  
See also G. Falcone, P. Plaumann and K. Strambach 
\cite{falconeetal} and D. Dikranjan and A.G. Bruno \cite{dikranjanbruno}.  
Some examples are also given by J.W. Nienhuys \cite{nienhuys}.

Following \cite{hewittross} 
we say a topological group $G$ is monothetic if it contains an element $\alpha$ such that the closure of $(n \alpha )_{n=1}^{\infty}$ is $G$.  Evidently by approximation, for arbitrary monothetic $G$,  the density and commutativity of $(n \alpha )_{n=1}^{\infty}$ implies the commutativity of $G$.  For a finite set of elements $\alpha _0, \alpha _1, \ldots , \alpha _k \in G$ set
$$
\rho (n) = \alpha _k n^k + \alpha _{k-1}n^{k-1} + \ldots \ + \alpha _1 n + \alpha _0,
$$
for each $n\in {\Bbb{N}}$.  
 Let $(p_n)_{n\geq 1}$ the sequence of rational primes.  Let $\lambda$ denote Haar measure on $G$ and let $L^p (G )$ denote the $L^p$ space of   $\lambda$-integrable functions on $G$.

\medskip
We  show the following.
\medskip
\begin{theorem}
\label{thm1}
Suppose $G$ is a compact, connected, monothetic group.   Then if one of the elements $\alpha _1, \ldots ,
\alpha _k  \in G$ is a generator and $f \in L^p (G)$ for $p> 1$, we have
$$
\lim _{N\to \infty} {1\over N} \sum _{n=1}^Nf(x + \rho (p_n) ) = \int _G fd\lambda,
$$
for almost all $x$ with respect to Haar measure on $G$.
\end{theorem}
\medskip 
If we drop the requirement that $G$ is connected the situation is different.  We now consider a class of totally disconnected monothetic groups of arithmetic character, of which the $p$-adic numbers is a special case.

 Let $a=(a_i)_{i\in \Bbb{Z}}$ denote a doubly infinite sequence of integers each greater than $1$.  Set $$\Bbb{Q}_a := \Pi _{i\in \Bbb{Z}} \{ 0, 1, \ldots , a_i-1\},$$ i.e. the space of doubly infinite sequences $a= (a_i)_{i\in \Bbb{Z}}$ with $a_i= 0$ for $i< n_o=n_o(a)$ for some $n_o$.

\par
Without loss of generality, assume for $x,y \in {\Bbb{Q}}_a$  that $n_o(x)=n_o(y) =0$ and that ${x} \ = \ (x_n)_{n=0}^{\infty}$ and 
${y }\ = \ (y_n)_{n=0}^{\infty}$ let
${z }\ = \ (z_n)_{n=0}^{\infty}$ be defined as follows.
Write $x_0 \ + \ y_0 \ = \ t_0 a_0 \ + \ z_0$, where
$z_0 \ \in \ \lbrace 0,1,...,a_0-1\rbrace$ and $t_0$ is a
rational integer.  Suppose $z_0,\cdots ,z_k$ and
$t_0,\cdots ,t_k$ have been defined.  Then write $x_{k+1} 
\  + \ y_{k+1} \ + \ t_k \ = \ t_{k+1} a_{k+1} \ + \ z_{k+1}$, where
$z_{k+1} \ \in \ \lbrace 0,1,...,a_{k+1}-1\rbrace$ and
$t_{k+1}$ is a rational integer.
We have thus inductively defined the sequence 
${ z }\ = \  (z_n)_{n=0}^{\infty}$, which we deem to be
${x} \ + \ {y}$.  The binary operation $+$
which we call addition makes $\Bbb{Q}_{a}$ an
abelian group.
\par 
For each non-negative integer $k$ let 
$$
\Lambda _k \ = \ \lbrace x \in  \Bbb{Q} _{a}
 : x_n =0 \ if \ n < k \rbrace.
$$
These sets form a basis at ${ 0} \ = \ (0,0,\cdots )$
for a topology on  $\Bbb{Q} _{a}$.  With respect to this
topology $\Bbb{Q} _{a}$ is  compact and the group
operations are continuous making $\Bbb{Q} _{a}$ a
locally-compact $\sigma$-compact abelian topological group.  A second binary
operation called multiplication, denoted by $\times$ and
compatible with addition is defined as follows.  Let
${ u} \ = \ (1,0,0,\cdots )$. Note that $(n{ u}
)_{n=0}^{\infty}$ is dense in $\Bbb{Q} _{ a}$.  First
on $(n{ u})_{n=0}^{\infty}$ define 
$k_1 {u} \times k_2 {u}$ to be 
$k_1 k_2 { u}$. Deeming multiplication to be
continuous on $\Bbb{Q} _{a}$ defines it off
$(n{u})_{n=0}^{\infty}$.   The binary operations
addition and multiplication makes $\Bbb{Q} _{a}$ a
topological ring .  We call $\Bbb{Z}_a  := \Lambda _0$  the $a$-adic integers and it  is compact
sub-ring of  $\Bbb{Q} _{a}$. Let $\eta_{a_i}$ denoted the counting measure on $\{ 0,1, \ldots , a_i-1 \}$.
For each integer let $\mu  =  \otimes _{i\in \Bbb{Z}} \eta _{a_i}$ denote the corresponding product measure on 
$\Pi_{n\in \Bbb{Z}}\{ 0, \ldots , a_n -1 \}$, defined first on "cylinder sets".  By cylinder sets we mean sets of the form
$$
A=\{ (b_i)_{i\in \Bbb{Z}}: b_{i_1} = c_{i_1} , \ldots , b_{i_r} = c_{i_r} \}
$$
for finite $r$, fixed $\{ i_1, \ldots , i_r \} \subset \Bbb{Z}$ and specific $c_{i_1}, \ldots c_{i_r}$.  We then specify $\mu$  by setting 
$$
\mu (A) = (a_{i_1} \ldots a_{i_r})^{-1}.
$$
on cylinder.  We then extend and complete to all Haar measurable sets on the topological group $\Bbb{Q}_{\bf a}$.

\par
The dual group of  $\Bbb{Z} _{a}$, which we denote
$\hat {\Bbb{Z} }_a$
consists of all rationals $t\ = \ {\ell \over A_r}$ where
$A_r \ = \ a_0\cdots a_r$ and $0\ \leq \ell \ \leq A_r$ for
some non-negative integer $r$.  To evaluate a character
$\chi _t$ at ${x}$ in  $\Bbb{Z} _{a}$  we write
$$
\chi_t ({x } ) \ = \ 
e\left ({\ell \over A_r}(x_0 + a_0x_1 +\cdots  + a_0\cdots
a_{r-1}x_r) \right ),
$$
where as usual, for a real number $x$, $e(x)$ denotes
$e^{2\pi i x}$.  We note that $\hat {\Bbb{Q}}_a = {\Bbb{Q}}_{a^*}$ where $a^*= ( a_{-n})_{n\in \Bbb{Z}}$.
\par 
Suppose at least one of the numbers 
$\alpha_1, \ldots , \alpha_k$ is a generator of 
$\Bbb{Z}_a$. 
Let $\delta_j: \Bbb{Z} \to \Bbb{Z}_a$ 
be the homomorphism defined by 
$\delta_{j} (n) = n \alpha_j$ $(j=1,2, \ldots , k)$ and let 
$\epsilon_j : \hat {\Bbb{Z}_a} \to \Bbb{Q}_a$ 
denote its adjoint.  
Hence for all $n\in \Bbb{Z}$ we have
$$e^{\epsilon_{j} 
\left ({l \over a_0 ... a_r} \right ) }
=
\chi_{l \over a_0  ...a_r }
\left (\delta_{j}(n )\right ) 
=
 \chi_{l \over a_0  ...a_r}(n \alpha_j).
$$
We can identify $\hat{\Bbb{Z}}_a$ with the 
quotient of $\Bbb{Q}_{a^*}$ by 
$A(\Bbb{Q}_{a^*} , \hat{\Bbb{Z}}_a )$ 
the annihilator of ${\hat{\Bbb{Z}}_a}$ in 
$\Bbb{Q}_{a^*}$ 
(\cite{hewittross}, 
Lemma 24.5).  
Let $\Psi :\Bbb{Q}_{a^{*}} \to \hat{ \Bbb{Z}}_a$ 
be the associated quotient map.  
Consider the map $G$, dependent on 
$\alpha_1, \ldots \alpha _k$, mapping
$\hat{ \Bbb{Z}}_a$ to $\Bbb{C}$, given by
$$
G\left (\chi _{l\over a_0 ... a_r} \right )= 
{1\over \phi ( D_r)}
\sum _{m=1\atop (m, D_r)=1}^{D_r}
e\left ( {\gamma (m) \over D_r} \right ),
$$
for all  ${l\over a_0 ... a_r} \in 
{\hat {\Bbb{Z}_a}}$.  
Here, $\phi$ is the Euler totient function. 
The integers $D_r$ and polynomial $\gamma$ are 
described as follows.  Let
$$l_j= l( \alpha_{j}(0)+\alpha_{j}(1) a_o
+ \alpha_{j} (2) a_{o} a_{1} 
+\ldots +\alpha_{j}(r-1) a_{0} \ldots a_{r-2}). \qquad 
(j=1,2, \ldots , k).
$$
Here $\alpha_{j} (r)$ denote the $r^{th}$ 
term of $\alpha_{j}$, viewed as a sequence. 
In general ${l_j \over a_0 \ldots a_r}$ 
is not in reduced form.   
Let ${m_j \over B_j}$ with $(m_j, B_j)=1$ 
denote ${l_j \over a_0a_1 \ldots a_r}$ in 
reduced terms.  
Let $D_r$ denote the least common multiple of 
the numbers $B_1, \ldots , B_k$ and 
define $\gamma$ via the identity
$${\gamma (x) \over D_r}
=
 {m_k \over B_k}x^k +  \ldots + {m_1 \over B_1}x.
$$
Define $m:\Bbb{Q}_{a^*} \to \Bbb{C}$ 
by $m(\chi ) =G (\Psi (\chi ))$ for all 
$\chi \in \Bbb{Q}_{a^*}$.  
Henceforth $\Bbb{F}(f)$ denotes the Fourier 
transform of $f$.

\bigskip
\begin{theorem}
\label{thm2}
Suppose that $p \in (1, 2]$ and that $f \in L^p \cap L^2 ( \Bbb{Q}_a )$.   Also suppose
$$
\rho (n) = \alpha _k n^k + \alpha _{k-1}n^{k-1} + \ldots \ + \alpha _1 n + \alpha _0,
$$
has degree at least $2$, with one of the numbers $\{ \alpha _1 , \ldots , \alpha _k \} \subset \Bbb{Q}_a$  a generator of $\Bbb{Q}_a$.   Then
$$
\lim _{N\to \infty} {1\over N} \sum _{n=1}^Nf(x + \rho (p_n) ) = \ell (f)(x),
$$
for $\ell (f) \in L^p$ almost everywhere with respect to Haar measure on $\Bbb{Q} _a$, where $$\Bbb{F}(\ell (f)) (\chi ) = m(\chi) \Bbb{F}(f)(\chi )$$
for almost all $x$ with respect to Haar measure on $\Bbb{Q}_a$.
\end{theorem}

\bigskip
\begin{theorem}
\label{thm3}
Suppose that $p \in (1, 2]$ and that $f \in L^p ( \Bbb{Z}_a )$.   Also suppose that
$$
\rho (n) = \alpha _k n^k + \alpha _{k-1}n^{k-1} + \ldots \ + \alpha _1 n + \alpha _0,
$$
has degree 
at least $2$, with one of the numbers 
$\{ \alpha _1 , \ldots , \alpha _k \} \subset \Bbb{Z}_a$ a generator of $\Bbb{Z}_a$.  Then
$$
\lim _{N\to \infty} {1\over N} \sum _{n=1}^Nf(x + \rho (p_n) ) = \ell (f)(x),
$$
for $\ell (f) \in L^p(\Bbb{Z}_a)$ almost everywhere with respect to Haar measure on $\Bbb{Z} _a$.  Here $\Bbb{F}( \ell (f)) ( \chi ) = G(\chi ) \Bbb(f)( \chi )$
for almost all $x$ with respect to Haar measure on $\Bbb{Z}_a$.
\end{theorem}
\bigskip
A measurable transformation $T: X \to X$ of a measure space $(X, \beta , \mu )$ is called measure preserving if $\mu ( T^{-1}(A)) = \mu (A)$ for all $A \in \beta$.   Here $T^{-1}(A)) := \{ x \in X : Tx \in A \}$.  For a measure space $(X, \beta , \mu )$, let $T_1, \ldots , T_k$ be a finite case set of commuting measure preserving transformations. Given $f \in  L^p (X, \beta , \mu )$ if $\pi _N$ denotes the number of prime numbers in the interval
$[1, N]$, for $(N=1,2, \ldots )$ we set
\begin{equation}
\label{eq1_1}
A_Nf(x):={1\over \pi _N } \sum _{1\leq p \leq N}f(T_1^{p}T_2^{p^2} \ldots T_k^{p^k} x). 
\end{equation}
Here the summation parameter $p$ runs over the 
primes.  The pointwise convergence of 
these averages is proved in 
\cite{trojan}.  
This is the new tool, 
that enables us to prove 
Theorems \ref{thm1} -- \ref{thm3} 
going beyond the results 
in \cite{nair3} and 
\cite{nair4}.
\bigskip

{\sl Some background :}  A sequence $(x_n)_{n\geq 1}$ in a compact topological group is said to be uniformly distibuted on $G$ if for each continuous functions $f: G \to {\Bbb{C}}$ we have 
$$
\lim _{N\to \infty} {1\over N} \sum _{n=1}^Nf(x_n ) = \int _Gfd\lambda .
$$
An example is $x_n= \{ n\alpha \} \ (n=1,2, \ldots )$, where $\alpha$ is an irrational number and for a real number $y$ we have used $\{ y \}$ to denote its fractional part.  This early result was due to  P. Bohl, W. Sierpinski and H. Weyl independently around the start of the 20th century. See the notes in 
\cite{kuipers}
for the historic background.  
Subsequently it was shown by 
H. Weyl \cite{weyl} that of 
$\rho (x) = \alpha _0 + \alpha _1 x + \ldots 
+ \alpha _kx^k$ 
and one of the numbers $\{ \alpha _1, \ldots , 
\alpha _k \}$ is irrational 
then the sequences $x_n= \rho (n) \ 
( n=1,2, \ldots )$ is uniformly distributed 
modulo $1$. 
Later I. M. Vinogradov and G. Rhin 
\cite{rhin}
proved that $x_n= \rho (p_n) \ ( n=1,2, \ldots )$ 
is uniformly distributed 
modulo $1$.  
Another result with a form similar in statement 
is that the result of 
Bohl, Sierpinski is that if $f$ is Lebesgue integrable on $[0,1)$ then
\begin{equation}
\label{eq1_2}
\lim _{N\to \infty}{1\over N} \sum _{n=1}^Nf(\{  x +  \alpha n \} ) = \int _0^1 f(t)dt , 
\end{equation}
almost everywhere with respect to Lebesgue 
measure.   This is an immediate consequence of 
Birkhoff's pointwise ergodic theorem and the 
ergodicty of the Lebesgue measure preserving map 
$T:x \to x+ \alpha$ for irrational $\alpha$ 
on $[0,1)$.  
An issue addressed by a number of authors is 
whether $(\{ n\alpha \})_{n\geq 1}$  in
\eqref{eq1_2} 
can be replaced by either 
$( \rho (n) )_{n\geq 1}$  or 
$( \rho (p_n) )_{n\geq 1}$ possibly under 
additional conditions on $f$.  
See for instance \cite{koksmasalem}, 
where the following 
is shown.  
Suppose  that  
$\rho (n) = \alpha _0 + \alpha _1 n+ \ldots , \alpha _k n^k$ with $(\alpha _1 , \ldots , \alpha _k )
 \notin {\Bbb{Q}}^k$. 
Also for 
$f \in L^2 ([0,1))$  with Fourier coefficients $(c_n)_{n\in {\Bbb{Z}}}$  suppose for some $\gamma > 0$  that 
\begin{equation}
\label{eq1_4}
\sum _{|n|\geq N} |c_n|^2 =O(\log _e N)^{\gamma}). 
\end{equation}
Then
\begin{equation}
\label{eq1_5}
\lim _{N\to \infty}{1\over N} \sum _{n=1}^Nf(\{  x +  \rho (n) \} ) = \int _0^1 f(t)dt , 
\end{equation}
almost everywhere with respect to Lebesgue 
measure on $[0,1)$.  

The second author was introduced to this topic 
when asked in private communication independently 
and at different times by R.C. Baker and M. 
Weber, whether condition \eqref{eq1_4} 
could be weakened or removed.
It turns out the answer is yes under the 
additional assumption 
that $f\in L^p([0,1))$ for $p>1$ and indeed in 
more general settings than $G=[0,1)$.  
This is the content of the papers 
\cite{nair3}, \cite{nair4}, 
and the current paper.  
The condition $p\not= 1$ is essential however 
because of the results in \cite{buczolichmauldin}.  
Notice the $x$ in \eqref{eq1_2} 
cannot in general be chosen to be $0$.  
This is the implication of 
J. Marstrand's famous result \cite{marstrand},
that there exists a $G_{\delta}$ set $B$ such 
that if $f(x) = I_{B}(x)$, i.e. 
the indicator set 
of $B$, then the limit in 
\eqref{eq1_2} 
fails to exist 
with $x=0$ for almost all $\alpha$.

\bigskip
\section{Some lemmas }
\label{S2}   
The following is a special case of S. Sawyer's Theorem \cite{sawyer}.

\bigskip

\begin{lemma}
\label{lemma4}
The pointwise convergence of the averages $(A_Nf)_{N \geq 1}$ in 
\eqref{eq1_1}
for $p>1$ implies that if 
$$
Mf(x) = \sup _{N\geq 1}\left |A_Nf(x)  \right |,  
$$
then there exists $C > 0$ such that $||Mf||_p \leq C||f||_p$.
\end{lemma}

\medskip

\begin{lemma}
\label{lemma5}
For each $\chi =\chi _{l\over a_o \ldots a_r} \in \Bbb{Z}_{a^*}$  we have
$$
G(\chi ) = \lim _{N\to \infty}{1\over \pi _N} \sum _{1\leq  p \leq N} \chi (\rho ( p)),
$$
where $p$ runs over the prime numbers in $[1,N]$.  Further for $\chi \in \Bbb{Q}_{a^*}$ we have
$$
m(\chi ) = \lim _{N\to \infty} {1\over \pi _N} \sum _{1\leq p \leq N} \chi ( \rho (p) ).
$$
\end{lemma}

\begin{proof}
Note that
$$
\chi _{l\over a_0 \ldots a_r}(\rho (p_n)) = \chi _{l\over a_0 \ldots a_r}( \alpha _0+ \alpha _1p_n + \ldots + \alpha _k p_n^k)
$$
 $$
= \prod _{j=0}^k \chi _{l\over a_0 \ldots a_r}(\alpha _j )^{p_n^j}
$$ 
$$
={\prod} _{j=0}^k 
\left (e^{2\pi i {l\over a_o \ldots a_r}^{p_n^j(\alpha _j(0)+\alpha _j(1)a_1+ \ldots  + \alpha _j (r-1)a_0a_1 \ldots a_{r-2})}}\right )
$$
$$
= \prod _{j=0}^k  \left ( e^{2\pi i {l_j\over a_o \ldots a_r}} \right  )^{p_n^j}
 =\prod _{j=0}^k  \left  (e^{2\pi i m_j\over B_j} \right )^{p_n^j}  =
 e^{2\pi i{\gamma (p_n) \over D_r}}.
$$
By splitting the primes into their reduced residue classes modulo $D_r$ we have
$$
{1\over \pi _N} \sum _{1\leq p \leq N} \chi _{l\over a_0 \ldots a_r}(\rho (p)) =
{1\over \pi _N} \sum _{1\leq p \leq N}    e^{2\pi i{\gamma (p) \over D_r}}. \eqno (N=2,3, \ldots )
$$
Let $\Lambda: \Bbb{N} \to \Bbb{R}$ denote the Von Mangoldt function defined by $\Lambda (n) = \log _e p$ if $n=p^l$ for some prime $p$ and positive integer $l$, and zero otherwise.  Using partial summation we see that
$$
{1\over \pi _N} \sum _{1\leq p \leq N} e^{2\pi i{\gamma (p) \over D_r}} =
{1\over N} \sum _{1\leq n \leq N} \Lambda (n) e^{2\pi i \gamma (n)\over D_r} +O(( \log _e N)^{-1} ).
$$
The Siegel -Walfish prime number theory for arithmetic progressions, says that for a fixed positive $u$, if $1 \leq D_r \leq ( \log _e N)^u$ and $(m, D_r)=1$, then for some $C> 0$
$$
 \sum _{1 \leq n \leq N \atop n\equiv m \bmod D_r} \Lambda (n) = {N \over \phi (D_r)} + o(Ne^{-C(\log _e N)^{1\over 2} }).
$$
Now note that
$$
\sum _{1 \leq n \leq N} \Lambda (n)
e^{2\pi i \gamma (n)\over D_r} 
= \left(  \sum _{m=1 \atop (m,D_r)=1}^{D_r}
e^{2\pi i{\gamma (m) \over D_r}} \right) 
\left (  \sum _{1 \leq n \leq N \atop n\equiv m\bmod D_r} \Lambda (n)  \right )
$$
$$
+O \left ( \sum _{p ^l \leq N; p|D_r}  \Lambda (p^l)e^{2\pi i \gamma (p^l)\over D_r} \right )
$$
Using the fact that the sum inside the $O$ 
notation is $O((\log _e N) (\log _e \log N))$, 
and the Siegel-Walfish theorem
$$
{1\over N}  
\sum _{1 \leq n \leq N \atop n\equiv m \bmod D_r} \Lambda (n)
e^{2\pi i \gamma (n)\over D_r} 
=
\left( {1\over \phi (D_r)} 
\sum _{1 \leq n \leq N} \Lambda (n)
e^{2\pi i \gamma (n)\over D_r}  \right)
$$
$$
+ O \left ({(\log _e N) (\log _e \log _e N) \over N} \right ),
$$
which proves the first part of 
Lemma \ref{lemma5}.   
The second part follows from the 
observation for all $x \in \Bbb{Z}_a$ 
and $\chi  \in \Bbb{Q}_{a^*}$ we have 
$\chi (x) = \Psi (\chi )(x)$. 
\end{proof}

\bigskip
\section{Proof of Theorem \ref{thm1} }
\label{S3}  
For the averages \eqref{eq1_1},  
let $X$ be $G$, let $\mu$ be Haar measure 
$\lambda$  on $G$ and let $\beta$ denote 
Haar measurable sets in $G$.  
Then if we set $T_i= x+ \alpha _i$ for 
$i=1, \dots , k$, for $N=1,2,\ldots $ and set
$$
a_N(f,x) = A_Nf(x) :={1\over  \pi_N } \sum _{1\leq p \leq N}f(x+ \rho (p)). 
$$
We wish to show $a_N(f,x)$ tends to 
$\int _Gfd\lambda$ almost everywhere 
with respect to Haar measure, as $N$ tends to 
$\infty$. 
\bigskip
First suppose $f:G \to \Bbb{C}$ is continuous.  
Because $G$ is a compact, 
connected 
and monothetic, for each non-trivial 
character $\chi$, and each 
generator $\chi (\alpha ) = e^{2\pi i\alpha _*}$ 
for an irrational number $\alpha_*$ 
(\cite{kuipers}, p. 275).
Thus
$\phi (\rho (n)) = e^{2\pi i \rho_{*}(n)}$, for a polynomial $\rho _*$ defined on the real numbers, with at least one irrational coefficient other than the constant term.  This means that
$$
{1\over N} \sum _{n=1}^N \chi (\rho (p_n))= {1\over N} \sum _{n=1}^N e^{2\pi i \rho _* (p_n)}. \eqno (N=1,2, \ldots )
$$
Now, if $\{ y \}$ denotes the fractional part of a real number $y$, then the sequence $(\{ \rho _* ( p_n) \})_{n\geq 1}$ is uniformly distributed modulo $1$.  Hence by Weyl's criterion on $G$, $a_N(f, x)$ tends to $\int _Gfd\lambda $ for all continuous $f: G \to \Bbb{C}$.
We now deal with the general case of Theorem 1.  Suppose $(f_n)$ is a sequence of continuous functions converging to fixed $f \in L^p$.  Choose $(n_k)_{k\geq 1}$ such that
$$
\sum _{k\geq 1}\int _G |f-f_{n_k}|^pd\lambda < \infty .
$$
This means
$$
\sum _{k\geq 1} |f-f_{n_k}|^p < \infty ,
$$
almost everywhere with respect to Haar measure on $G$.  Thus for each $\epsilon > 0$, there exists a sequence of functions 
$(f_{\epsilon , k})_{k\geq 1}$ such that $||f-f_{\epsilon , k}||_p^p \leq \epsilon ^{2k}$ and $f_{\epsilon , k}$ tends to $f$ as $k$ tends to infinity, almost everywhere with respect to $A$ with respect to Haar measure.  Let
$$
m(f) = \sup _{N\geq 1}|\alpha _N f|.
$$
Notice that $m$ is sub-additive  i.e.
$$
m(f+g) \leq m(f) + m(g).
$$
Let 
$$
E_{\epsilon ,k} = \{ x \in G : m(f-f_{\epsilon , k})(x)> e^{k\over p} \} .
$$
Notice that,
$$
\mu (E_{\epsilon , k}) \leq \left ( {1\over \epsilon } \right )^k \int _{E_{\epsilon , k}}[m(f-f_{\epsilon , k})(x)]^p d\lambda 
$$
$$
\leq  \left ( {1\over \epsilon } \right )^k|| m(f-f_{\epsilon , k})(x)||_P^P,
$$
which, using the fact that $||Mf||_p \leq C_p||f||$, is 
$$
\leq  c\left ( {1\over \epsilon } \right )^k|| ((f-f_{\epsilon , k})(x)||_P^P,
$$
$$
\leq  c \left ( {1\over \epsilon } \right )^k .\epsilon ^{2k} = C\epsilon ^k.
$$
Now
$$
a_N(f, x) =a_n(f-f_{\epsilon , k}, x) + a_N(f_{\epsilon , k}, x).
$$
This means that
$$
\left |a_N(f, x) -\int _Gfd\lambda | \leq |a_N(f-f_{\epsilon , k}, x)| +|a_N(f_{\epsilon , k}, x) - \int_Gfd \lambda \right |
$$
We have already shown that
$$
\lim _{N \to \infty}a_N (f_{\epsilon , k}, x) = \int_G f_{\epsilon, k}d\lambda ,
$$
almost everywhere with respect to Haar measure on $G$.  Therefore
$$
\limsup _{N\to \infty} \left  |a_N(f, x) -\int _Gfd\lambda \right | \leq \limsup _{N\to \infty} \left |a_N(f-f_{\epsilon , k}, x) \right | +
$$
$$
+ \left  | \int _G fd\lambda  - \int _G f_{\epsilon , k}d\lambda  \right |
$$
$$
\leq m(f- f_{\epsilon , k})+ \int _G\left  |f-f_{\epsilon , k}\right | d\lambda  .
$$
Thus as $N$ tends to infinity we know $a_N(f, x)$ tends to $\int _G fd\lambda $ for all $x$ in $A_{\epsilon}= \cup _{k\geq 1} E_{\epsilon , k}$. Let $B_{\epsilon}$ denote the null set of which $f_{\epsilon , k}$ tends to $f$ as $k$ tends to infinity. This means that
$$
\mu (A_{\epsilon}\cup B_{\epsilon}) \leq \sum _{k\geq 1} \mu (E_{\epsilon , k}) \leq C \sum _{k\geq 1}\epsilon ^k = {C\epsilon \over 1- \epsilon}.
$$
Letting $\epsilon \to 0$ completes the 
proof of Theorem \ref{thm1}, 
upon noting that 
$\pi _N \sim {N\over \ln (N)}$ by the 
prime number theorem. 

\bigskip
\section{Proof of Theorems \ref{thm2} and \ref{thm3} }
\label{S4}  
Fix $p\in (1, \infty)$, 
let $f \in L^p(\Bbb{Q}_a)$, and suppose the 
support of $f$ is contained  in $\Lambda _k$ for 
some non-positive integer $k$.  
This means $f = f I_{\Lambda _k}$, where  for a 
set $A$, we let $I_A$ denotes the indicator 
function.  
Applying the Stone-Weierstrass theorem shows that 
$f$ can be approximated in the $L^p$ norm 
functions of the form 
\begin{equation}
\label{eq4_1} 
I_{\Lambda _k} \sum _{j=1}^{\nu } b_j \chi _j, 
\end{equation}
with each $b_j \in \Bbb{C}$ and $\chi \in \Bbb{Q}_{a^*}$.  
Because compactly supported functions are dense in $L^p( {\Bbb{Q}_a })$, it follows functions of the form are also dense in  $L^p( {\Bbb{Q}_a })$.  
It is clear for each $N$  that the averaging operator $A_N$  commutes with translations on $\Bbb{Q}_a$.  
We know by Lemma \ref{lemma4} 
that the function $\ell f(x)$ exists almost everywhere.  
Also as a consequence of Lemma \ref{lemma4}  
and the Lebesgue dominated covergence theorem, 
it follows 
that the functions $(A_Nf(x))_{N\geq 1}$ also 
converges in $L^p$ to the same limit.  
Evidently the operator $\ell $ commutes with 
translations on $\Bbb{Q}_a$.  
It follows that the linear operator $\ell $ is a 
Fourier multiplier on $L^p(\Bbb {Q}_a )$.  
Hence there is a bounded measurable 
function $m$ on $\Bbb{Q}_a$ such that 
for all $f \in L^p \cap L^2( \Bbb{Q}_a )$  
we have
$$
\Bbb{F} ( \ell f)\ (\chi ) = m(\chi ) \Bbb{F}(f)(\chi ),
$$
almost everywhere in $\Bbb{Q}_a$.  To identify $m$ we evaluate $\ell$ on functions of the form (4.1) we note that
$$
\Bbb{F}(f)(\xi) = \sum _{j=1}^{\nu } b_j \Bbb{F}(\chi _j I _{\Lambda _k})(\xi ) = \sum _{j=1}^{\nu } b_j \Bbb{F}(  I _{\Lambda _k})
(\xi+ \chi _j )
$$ 
\begin{equation}
\label{eq4_2}
= \lambda (\Lambda _k) \sum _{j=1}^{\nu } b_j  I _{  \chi _j +A(\Bbb{Q}_{a^*}, \Lambda _k ) }(\xi) ,  
\end{equation}
where $A(\Bbb{Q}_{a^*}, \Lambda _k )$ denotes the annihilator in $\Bbb{Q}_{a^*}$ of $\Lambda _k$.  The penultimate identity follows from the fact that multiplication by a character shifts the Fourier transform.  The last identity follows from the identity
\begin{equation}
\label{eq4_3}
\Bbb{F} \left   ( 
{1\over \lambda  (\Lambda _k)} 
 I_{\Lambda _k }  
\right )   
= I _{ A(\Bbb{Q}_{a^*}, \Lambda _k ) }.
\end{equation}

Indeed if $\chi \in \Bbb{Q}_{a^*}$ 
the restriction of $\chi$ to  
$\Lambda _k$ is plainly a continuous 
character of $\Lambda _k$ and 
\eqref{eq4_3} 
follows since the Haar integral on $\Lambda _k$ of any non-trivial character of $\Lambda _k$ is zero 
(\cite{hewittross}, Lemma 23.19, p. 363).
For $f$ as in \eqref{eq4_1} 
we have
$$
A_Nf(x) ={1\over N} \sum _{n=1}^N I_{\Lambda _k} (x + \rho (p_n) \sum _{j=1}^{\nu} b_j \chi _j(x + \rho (p _n )) .
$$
where $(p_n)_{n\geq 1}$ denotes the strictly increasing sequence of all primes.  If $x$ is not in $\Lambda _k$ then nor is $x+ \rho (p_n)$ because
$\rho (p_n) \in \Lambda _k$.  This means that we have
\begin{equation}
\label{eq4_4}
A_Nf(x) = A_Nf(x) I_{\Lambda _k}(x). 
\end{equation}
For $x \in \Lambda _k$ we have 
$$
A_Nf(x) ={1\over N} \sum _{n=1}^N  \sum _{j=1}^{\nu} b_j \chi _j(x + \rho (p _n )) .
$$
$$
 = \sum _{j=1}^{\nu} \chi (x)  {1\over N} \sum _{n=1}^N \chi _j(\rho (p _n )).
$$
By Lemma \ref{lemma5} and 
\eqref{eq4_4} 
for all $x \in \Bbb{Q}_a$ we get
\begin{equation}
\label{eq4_5}
\ell (f) (x)= \lim _{N\to \infty}I_{\Lambda _k}(x)A_Nf(x)=I_{\Lambda _k}(x) \sum _{j=1}^{\nu} b_jG(\Psi (\chi _j)) \chi _j (x)
\end{equation}
$$
=I_{\Lambda _k}(x) \sum _{j=1}^{\nu} b_jG(\Psi m(\chi _j )) \chi _j (x).
$$
Computations similar to 
\eqref{eq4_2} and \eqref{eq4_5},
show for all $\chi \in \Bbb{Q}_{a^*}$
\begin{equation}
\label{eq4_6}
\Bbb{F}(\ell f)(\chi ) = \lambda ( \Lambda _k) \sum _{j=1}^{\nu} b_jG(\Psi (\chi _j))I_{\lambda _j+ A(\Bbb{Z}_{a^*} , \Lambda _k)}(\chi ).
\end{equation}
If $\chi \in \chi _j + A(\Bbb{Q}_{a^*} , \Lambda _k)$ then $\chi = \chi _j + \chi '$ with $\chi '\in A(\Bbb{Q}_{a^*} , \Lambda _k)$.  Consequently $\Psi (\chi ) = \Psi (\chi _j ) + \Psi ( \chi ')$.  Recall $\Psi : \Bbb{Q}_{a^*} \to  {\Bbb{Q}_{a^*} \backslash A(\Bbb{Q}_{a^*} , \Lambda _k)}$ and $\chi '$ is in $A(\Bbb{Q}_{a^*} , \Lambda _k)$, which is a subset of $A(\Bbb{Q}_{a^*} , \Bbb{Z}_a)$ because $\Bbb{Z}_a$ is a subset of $\Lambda _k$.  Hence if $\chi \in \chi _j + A(\Bbb{Q}_{a^*} , \Bbb{Z}_a)$,
then $\Psi (\chi ) = \Psi (\chi _j )$.  
Using \eqref{eq4_2} 
and 
\eqref{eq4_4} 
we have that for all $x \in \Bbb{Q}_{a^*}$
$$
\Bbb{F} (\ell (f)) (x)(\chi ) =m(\chi ) \Bbb{F}(f)(x)(\chi ) ,
$$
establishing the theorem for all functions 
of the form \eqref{eq4_1}.
To prove it in general, given 
$f \in L^p \cap L^2( \Bbb{Q}_a )$,  
let $(f_n)_{n\geq 1}$ be a sequence of 
functions of the form \eqref{eq4_1}
which 
converge to $f \in L^2$.  
Because $\ell$ is continuous  
$( \ell (f_n))_{n\geq 1}$ converges 
to $\ell (f)$ in $L^2(\Bbb{Q} _a)$.  
Hence a  subsequence of 
$(\Bbb{F} ( \ell (f_n)) _{n\geq 1}$  converges  
almost everywhere in $\Bbb{Q}_{a^*}$ to 
$\Bbb{F}( \ell (f))$. The desired conclusion 
follows from the fact that 
$\Bbb{F} ( \ell (f_n)) = m( \chi ) \Bbb{F} (f_n)$ 
and a subsequence of $(\Bbb{F} (f_{n}))_{n\geq 1}$ 
converges to $\Bbb{F} (f)$ almost everywhere.  

\bigskip
To prove Theorem \ref{thm3}, 
we need only observe that $f \in L^p(\Bbb{Z }_{a})$ then it may be approximated by functions in
$f \in L^p(\Bbb{Z }_{a})$ of the form 
\eqref{eq4_1} 
with $k=0$ as required.

\bigskip \bigskip
\section{More on the Limit Function 
${\bf \ell (f)}$}
\label{S5}  
Define $H: \hat {\Bbb{Z}}_a \to \Bbb{C}$ , dependent on $\alpha _1, \ldots \alpha _k$  by 
$$ 
H  \left ({l \over a_0 \ldots a_k} \right ) : ={1\over D_k} 
\sum _{m=1}^{D_k}e\left ( {\gamma (m) \over D_k} \right ).
$$
Also define $n:\Bbb{Q}_{a^*} \to \Bbb{C}$ by $n(\chi ) : =H( \Psi (\chi ))$.  
In \cite{asmarnair} 
we show the following lemma and the two theorems
\medskip
\medskip

\begin{lemma}
\label{lemma6}
For each $\chi =\chi _{l\over a_o \ldots a_r} \in \Bbb{Z}_{a^*}$  we have
$$
H(\chi ) = \lim _{q\to \infty}{1\over N} \sum _{1\leq q  \leq N} \chi (\rho ( q)).
$$  Further for $\chi \in \Bbb{Q}_{a^*}$ we have
$$
n(\chi ) = \lim _{N\to \infty} {1\over N} \sum _{1\leq q  \leq N} \chi ( \rho (q) ).
$$
\end{lemma}
\medskip
We also showed the following theorems.
\bigskip
\begin{theorem}
\label{thm7}
Suppose that $p \in (1, 2]$ and that $f \in L^p \cap L^2 ( \Bbb{Q}_a) $.   
Also suppose that 
$$
\rho (n) = \alpha _k n^k + \alpha _{k-1}n^{k-1} + \ldots \ + \alpha _1 n + \alpha _0,
$$
has degree at least $2$ with one of the numbers $\{ \alpha_1, \ldots , \alpha _k  \} \subset \Bbb{Q}_a $ a generator of $\Bbb{Q}_a$.  Then
$$
\lim _{N\to \infty} {1\over N} \sum _{n=1}^Nf(x + \rho (n) ) = k (f)(x),
$$
for $ k(f) \in L^p$ almost everywhere with respect to Haar measure on $\Bbb{Q} _a$, where $$\Bbb{F}( k(f)) (\chi ) = n(\chi) \Bbb{F}(f)(\chi ). $$
\end{theorem}

\bigskip
\begin{theorem}
\label{thm8}
Suppose that $p \in (1, 2]$ and that 
$f \in L^p ( \Bbb{Z}_a )$.  
Also suppose
$$
\rho (n) = \alpha _k n^k + \alpha _{k-1}n^{k-1} + \ldots \ + \alpha _1 n + \alpha _0,
$$
has degree 
at least $2$  with one of the numbers $\{ \alpha_1, \ldots , \alpha_k \} \subset \Bbb{Z}_a $ a generator of $\Bbb{Z}_a$.   Then
$$
\lim _{N\to \infty} {1\over N} \sum _{n=1}^Nf(x + \rho (n) ) = k (f)(x),
$$
for $k (f) \in L^p(\Bbb{Z}_a)$ almost everywhere with respect to Haar measure on $\Bbb{Z} _a$, where $$\Bbb{F}( k (f)) ( \chi ) = n(\chi ) \Bbb(f)( \chi ).$$
\end{theorem}
\bigskip
The Riesz representation theorem tells us that there exist measures $\mu _{\ell}$ and $\mu _{k}$ on $\Bbb{Z}_a$ such that at least for $f\in C(\Bbb{Z}_a)$ we have $\ell (f) = \int fd\mu _{\ell} $ and $k (f) = \int fd\mu _{k} $.  The same is true for with $\Bbb{Z}_a$ replaced by $\Bbb{Q}_a$.   A natural question is whether the measures $\mu _{\ell}$ or $\mu _{k}$ are in fact Haar measure.  One might conjecture that if $\rho$ has degree at least two the answer is no.  A special case we can deal with is $\rho (n)=n^2$.
To see this, choose $l\not=0$ such that  $t={l/a_0\ldots a_s}$,  is of the form
$a\over q$ for a prime $q$, with $a\not= 0$.  In this case $H(\chi _{t})$ is of the form
$C \sum _{r=1}^q
e^{2\pi i ar^2\over q}$ for a non-zero 
constant $C$.   
As is well known 
$\left| \sum _{r=1}^q
e^{2\pi i a r^2 \over q} \right|
=
q^{1 \over 2}$ (\cite{apostol}, 
p. 166).  This means 
that the Fourier coefficient of 
the measure $\mu _{\ell}$  of the squares
on  $\Bbb{Z}_a$ is non-zero.  
For $\mu _{\ell}$ to be Haar measure you would 
need $H({t})=0$ for every non-zero $t$ which we 
have discounted.

Unfortunately, dealing with polynomials more 
general than $\rho (n)=n^2$ or $s>1$, is a good 
deal more complex and is as yet unresolved.  This 
is because it relies on lower bounds for 
exponential sums of the form $\left| 
\sum _{r=1}^q e^{2 \pi i a \rho(r) \over q}
\right|$, which seems a serious 
undertaking and yet to appear in the literature.  We can however show that $\mu _{\ell}$ and 
$\mu _{k}$ are always continuous which we now 
prove.

\bigskip
For a measure $\mu$ defined on the group 
$\Bbb{Z}_a$ set 
$$
\Bbb{F}(\mu ) (\chi_t ) =\int_{\Bbb{Z}_a} \overline{ \chi }_t d\mu , \qquad \quad 
( t \in \hat{\Bbb{Z}}_a)
$$
i.e. the Fourier transform of $\mu$. The analogue of Wiener's lemma on the group $\Bbb{Z}_a$ 
(\cite{grahammcgehee}, 
p. 236) requires us to show that
$$
\lim _{r\to \infty} {1\over A_r} \sum _{ \{ { l \over A_s} : 0\leq s \leq r \}} |\Bbb{F}(\mu ) (\chi _t )|^2 =0.
$$
\medskip
We need the following Lemma \cite{nair5}.

\begin{lemma}
\label{lemma9}
Suppose that
$$
\psi (x) \ = \ a_dx^d \ + \ \cdots \ + \ a_1x
$$
for integers $a_i \ (i\ = \ 1,2, \cdots ,d)$ and let

$$
S({\psi  | q}) \ = \ \sum _{r=0}^{q-1}e^{2\pi i\psi (r)q^{-1}}.
$$
Then there exist $\delta _0 \ > \ 0$, and 
$C_{ \delta _0} \ > \ 0$ such that
$$
|S(\psi  |q )| \ \leq \ {C_{\delta _0}\over q^{ \delta _0 }}.
$$ 
\end{lemma}

This means that there exist $\delta _H >0$ such that  $H({l\over A_r}) \ll A_r^{- \delta _H}$.  Also
$$
\sum _{1\leq m \leq q \atop (m,q)=1 } e( \rho (m)) =\sum _{1\leq m \leq q } e( \rho (m)) - \sum _{p\over q} \sum _{1\leq m \leq q \atop p|m } e( \rho (m)),
$$
so there exist $\delta _{\ell} >0$ such that $G({l \over A_r}) \ll A_r^{-\delta _{\ell}}$.  In either case there is $\delta > 0$ such that
$$
{1\over A_r} \sum _{ \{ t= {l \over A_s} : 0\leq s \leq r \}} |\Bbb{F}(\mu ) (\chi _t )|^2 \ll A_r^{-\delta},
$$
as required, where $\mu$ is either $\mu _{\ell}$ or $\mu_k$.  Hence both measures $\mu _k$ and 
$\mu _{\ell}$ are continuous on 
$\Bbb{Z}_a$. 

\bigskip

If we consider $\Bbb{Q}_a$, we notice that 
$m(\chi ) = G( \Psi (\chi ) )$ and 
$n(\chi ) = H(\Psi ( \chi )) $.  
The argument for $\Bbb{Z}_a$  with 
$G(\chi )$ replaced by $G( \Psi (\chi ) )$  
and with $H(\chi )$ replaced by 
$H( \Psi (\chi ))$ now implies that both 
$\mu _{\ell}$ and $\mu_k$ are continuous.  
\medskip
\medskip

{\bf Acknowledgement :} 
The authors thank Buket Eren G\"okmen 
for comments that much improved the 
readability of the manuscript.  
Radhakrishnan Nair thanks 
Laboratoire de Math\'ematique de
l'Universit\'e Savoie Mont Blanc 
for its hospitality and financial 
support while this paper was being written.

\end{document}